\documentclass[reqno, 11pt]{amsart} 

\usepackage{amssymb,bbm}
\usepackage{amsfonts}
\usepackage{amsmath, amscd}
\usepackage{graphicx}
\usepackage{color}
\usepackage{xcolor}
\usepackage{enumerate}
\usepackage{soul}
\usepackage{dsfont}
\usepackage{marginnote}

\usepackage{hyperref}
\hypersetup{
   colorlinks=true,
   linkcolor=blue,
}

\usepackage{tikz}
\usetikzlibrary{automata}
\usetikzlibrary{arrows,shapes,calc}
\usetikzlibrary{positioning}

\makeatletter
\def\l@subsection{\@tocline{2}{0pt}{2.5pc}{5pc}{}}
\renewcommand\tocchapter[3]{%
  \indentlabel{\@ifnotempty{#2}{\ignorespaces#2.\quad}}#3%
}
\newcommand\@dotsep{4.5}
\def\@tocline#1#2#3#4#5#6#7{\relax
  \ifnum #1>\c@tocdepth % then omit
  \else
    \par \addpenalty\@secpenalty\addvspace{#2}%
    \begingroup \hyphenpenalty\@M
    \@ifempty{#4}{%
      \@tempdima\csname r@tocindent\number#1\endcsname\relax
    }{%
      \@tempdima#4\relax
    }%
    \parindent\z@ \leftskip#3\relax \advance\leftskip\@tempdima\relax
    \rightskip\@pnumwidth plus1em \parfillskip-\@pnumwidth
    #5\leavevmode\hskip-\@tempdima{#6}\nobreak
    \leaders\hbox{$\m@th\mkern \@dotsep mu\hbox{.}\mkern \@dotsep mu$}\hfill
    \nobreak
    \hbox to\@pnumwidth{\@tocpagenum{#7}}\par
    \nobreak
    \endgroup
  \fi}
\makeatother
\AtBeginDocument{%
\makeatletter
\expandafter\renewcommand\csname r@tocindent0\endcsname{0pt}
\makeatother
}
\def\l@subsection{\@tocline{2}{0pt}{2.5pc}{5pc}{}}

\newtheorem{theorem}{Theorem}[section]

\newtheorem{lemma}[theorem]{Lemma}
\newtheorem{corollary}[theorem]{Corollary}
\newtheorem{definition}[theorem]{Definition}
\newtheorem{remark}{Remark}[section]

\newcommand{\Z}{\mathbb{Z}}

\makeatletter
\@addtoreset{equation}{section}
\makeatother

\renewcommand\thetable{\thesection.\@arabic\c@table}

\bibdata{prob}
\bibstyle{alpha}

\title[Survival of one dimensional renewal contact process]{Survival of one dimensional renewal contact process}

\author[R. Santos, M. E. Vares]{Rafael Santos$^1$, Maria Eul\'alia Vares$^2$}
\thanks{1. Instituto de Matem\'atica. Universidade Federal
do Rio de Janeiro, RJ, Brazil.
E-mail: rafaels@dme.ufrj.br. Supported by CNPq, Brazil grant 151736/2022-7}
\thanks{2. Instituto de Matem\'atica. Universidade Federal
do Rio de Janeiro, RJ, Brazil.
E-mail: eulalia@im.ufrj.br. Partially supported by CNPq, Brazil grant 310734/2021-5, and FAPERJ, Brazil grant CNE E-26/200.442/2023}

\date{\today}

\begin{document}

\begin{abstract}
The renewal contact process, introduced in $2019$ by Fontes, Marchetti, Mountford, and Vares, extends the Harris contact process in $\mathbb{Z}^d $ by allowing the possible cure times to be determined according to independent renewal processes (with some interarrival distribution $\mu$) and keeping the transmission times determined according to independent exponential times with a fixed rate $\lambda$. We investigate sufficient conditions on $\mu$ to have a process with a finite critical value $\lambda_c$ for any spatial dimension $d \geq 1$. In particular, we show that $\lambda_c$ is finite when $\mu$ is continuous with bounded support or when $\mu$ is absolutely continuous and has a decreasing hazard rate.


\end{abstract}

\maketitle

%\tableofcontents

%%%%%%%%%%%%%%%%%%%%%%%%%%%%%%%%%%%%%%%%%%%%%%%%%%%%%%%%%%%%%%%
%!TEX root = main.tex

\section{Introduction}
%%%%%%%%%%%%%%%%%%%%%%%%%%%%%%%%%%%%%%%%%%%%%%%%%%%%%%%%%%%%%%%
\label{intro}

The classical contact process, introduced by Harris \cite{Harris} in $1974$, is a model for the spread of infectious diseases and has been intensively studied (see for instance \cite{Durrett} and \cite{Liggett}). 
Several variants and extensions appeared in the literature, including the Renewal Contact Process, which is the model studied in this paper and was introduced in \cite{FMMV1}, \cite{FMV2}, motivated by questions regarding long range percolation.

In both models, the sites of $\mathbb{Z}^d$ represent individuals that can be healthy or infected, and the state of the population at time $t$ is represented by a configuration $\xi_t \in \{0,1\}^{\mathbb{Z}^d}$, where $\xi_t(x) = 0$ means that the individual $x$ is healthy at time $t$ and $\xi_t(x) = 1$ means that the individual $x$ is sick at time $t$. The classical model considers a Markovian evolution: Infected individuals become healthy at rate $1$ independently of everything else, and healthy individuals become sick at a rate equal to a given parameter $\lambda$ times the number of infected neighbors. The Renewal Contact Process extends this model by allowing the possible cure times to be determined by i.i.d. renewal processes with an interarrival distribution $\mu$ on $[0,\infty)$ such that $\mu\{0\} < 1$. This model will be denoted by RCP($\mu$). This additional flexibility comes with the cost of losing the Markov property, which brings new difficulties.

In \cite{Harris2}, the classical contact process was alternatively defined using a percolation structure, known as \textit{Harris graphical construction}, based on a system of infinitely many independent Poisson point processes. Besides being a very useful tool, Harris construction immediately suggests the consideration of more general percolation structures, with Poisson point processes being replaced by more general point processes. Natural variants include the RCP($\mu$) and the models in \cite{HUVV}, where, besides the contact process on dynamic edges introduced in \cite{LR}, the authors consider a version of RCP with renewals on the transmission marks. (See also \cite{FMUV}, where the results of \cite{FMV2} have been strongly refined, and \cite{F} for a short survey.)

For sake of completeness, we now recall the basic definition, based on two independent sequences of point processes:  

\noindent $\bullet$ $(\mathcal{R}_x)_{x \in \mathbb{Z}^d}$ is a family of independent renewal processes with interarrival distribution $\mu$. $\mathcal{R}_x$ describes the cure marks at $x$. When the specific vertex is not important we will denote simply by $\mathcal{R}$ the renewal process constructed in the same way as  $\mathcal{R}_x$.

\noindent $\bullet$ $(\mathcal{N}_{x,y})_{(x,y) \in \mathcal{E}}$ is a family of independent Poisson processes with rate $\lambda$. Here $\mathcal{E}$ denotes the set of ordered pairs of nearest neighbors in $\mathbb{Z}^d$. $\mathcal{N}_{x,y}$ describes the times when the vertex $x$ tries to infect $y$. 

 The RCP($\mu$) is described in terms of (time oriented) paths, also called {\em infection paths}.  A path from $(x,s)$ to $(y,t)$ for $x,y \in \Z^d$ and $s< t$ is a c\`adl\`ag function $\gamma: [s,t] \ \rightarrow \ \Z^d$ so that: (i) $\gamma (s)= x$;
(ii) $\gamma (t)= y$; (iii) $\forall u \in [s,t],  u \notin  \mathcal{R}_{\gamma (u)}$; and
(iv) $\forall u \in [s,t]$, if  $\gamma (u-) \neq \gamma (u)$, then $u \in  N_{\gamma (u-),\gamma (u)}$.  
In particular, if $A \subset \Z^d$, our RCP($\mu$) with infection parameter $\lambda$ and initially infected set $A \subset \Z^d$ is defined by
$$
\xi^A_t(y) \ = \ 1 \iff  \  \exists \mbox{ a path from } (x,0) \mbox{ to } (y,t), \mbox {for some }x \in A.
$$

Among the main questions of interest regarding RCP($\mu$) we want to better understand the survival and extinction of the infection depending on the parameter $\lambda$ and the distribution $\mu$. In this direction, given $\mu$, the critical parameter for RCP($\mu$) is defined as
\[
\lambda_c(\mu) := \inf\{\lambda: P(\tau^{\{0\}} = \infty) > 0\}, 
\]
\noindent where $\tau^{\{0\}} := \inf\{t: \xi_t^{\{0\}} \equiv 0\}$ (with the usual convention that $\inf \emptyset = \infty$). The articles \cite{FMMV1} and \cite{FGS} (this last one for finite volumes) obtained sufficient conditions on $\mu$ to assure that $\lambda_c(\mu) = 0$ (when $\mu$ has a heavy tail), while \cite{FMV2} obtained sufficient conditions on $\mu$ to assure that $\lambda_c(\mu) > 0$, and which were significantly relaxed in \cite{FMUV}. In particular, the authors in \cite{FMUV} proved that $\lambda_c(\mu) > 0$ whenever $\int x^\alpha \mu(dx) < \infty$ for some $\alpha > 1$. Even combining the results of \cite{FMMV1} and \cite{FMUV}, there remains an important gap where we do not know whether the critical parameter is zero or not. 

As we can see from the previous paragraph, the study of sufficient conditions to assure that $\lambda_c(\mu)$ is zero or positive is already intense. Nevertheless, it is also natural to ask whether $\lambda_c(\mu)$ is finite or infinite and this question was very little explored in the literature until now. Naturally if $\mu$ is degenerate (there exists $c>0$ such that $\mu\{c\}= 1$), the infection always dies out (at time $c$) and consequently $\lambda_c(\mu) = \infty$. It is natural to conjecture that except for this degenerate case (simultaneous extinction), the critical parameter should be finite, but we still do not have a proof for that in the literature. Since $\lambda_c(\mu)$ is non-increasing in $d$, it is enough to consider the one-dimensional case. When $d \geq 2$, this problem becomes much simpler since we can construct an infinite infection path using each vertex only once (i.e. through a coupling with supercritical oriented percolation), avoiding dependencies within each renewal process. This idea was explored in the proof of Theorem 1.3 (ii) in \cite{HUVV}, which states that the critical parameter is finite for a more general version of the renewal contact process. The proof for RCP($\mu$) is similar.

Considering $d=1$, and obviously excluding the cases where we already know that $\lambda_c(\mu) = 0$, the only general result that we are aware of regarding $\lambda_c(\mu)< \infty$  is the one described in Remark $2.2$ of \cite{FMUV}. It argues that if $\mu$ has a density and a bounded decreasing hazard rate, then $\lambda_c(\mu) $ is finite. The proof uses a coupling of the RCP($\mu$) with the classical contact process using the construction described in \cite{FMV2}. In particular, if $\mu$ and $\tilde{\mu}$ are absolutely continuous distributions on $[0,\infty)$ with hazard rates $h_{\mu}$ and $h_{\tilde{\mu}}$ such that $h_{\mu}(t) \leq h_{\tilde{\mu}}(t)$ for any $t \geq 0$ and moreover $h_{\mu}(t)$ is decreasing in $t$, then $\lambda_c(\mu) \leq \lambda_c(\tilde{\mu})$. Our Theorem \ref{theo-dfr} below improves this by allowing the hazard rate to be unbounded near the origin. This covers for example the case where the times between cure marks have a Weibull distribution with shape parameter smaller than one.

Theorem \ref{theo-bounded} is the other main result of this paper. It states that $\lambda_c(\mu)$ is finite whenever $\mu$ is continuous (i.e. $\mu\{x\}=0$ for all $x$) and has a bounded support.

\begin{theorem}
Let $\mu$ be a continuous distribution such that $\mu[0,b]=1$ for some $b < \infty$. Then $\lambda_c(\mu) < \infty$.
\label{theo-bounded}
\end{theorem}

\begin{theorem}
Let $\mu$ be an absolutely continuous distribution on $[0,\infty)$ with density $f$ and distribution function $F$ such that its hazard rate $\frac{f(t)}{1-F(t)}$ is decreasing in $t$. Then $\lambda_c(\mu) < \infty$. 

\label{theo-dfr}
\end{theorem}

In Section \ref{bounded} we have the proof of Theorem \ref{theo-bounded} and Section \ref{dfr} is dedicated to the proof of Theorem \ref{theo-dfr}. Both proofs involve coupling with supercritical oriented percolation. Differently from \cite{HUVV}, we need to find a way to do it using the same vertices more than once to construct an infinite path, and therefore it is necessary to control the dependencies that appear. Each proof addresses this problem in a different manner, and we still could not find a way to deal with it for any non-degenerate $\mu$, keeping us from proving the conjecture mentioned before. 

\section{$\mu$ continuous and with bounded support}
\label{bounded}

This section is devoted to proving Theorem \ref{theo-bounded}. Before this, we need to state and prove some useful lemmas and a few extra definitions are needed.

To begin, we state a lemma from \cite{HUVV} that gives a uniform estimate for renewal processes, and which will be important to many of the proofs presented in this section.

\begin{lemma}
\textbf{\cite[Theorem $4.1(ii)$]{HUVV}} Let $\mu$ be a continuous distribution on $[0,\infty)$ and let $\mathcal{R}$ be a renewal process with interarrival distribution $\mu$ started from some $\mathfrak{T} \leq 0$. Then, given $p_0 > 0$ there exists $w_0 = w_0(p_0) > 0$ such that uniformly on $\mathfrak{T}$ we have: %\marginnote{\blue por que não  $\mathfrak{T}\leq 0$ ?}
\[
\sup_{t \geq 0} P(\mathcal{R} \cap [t,t+w_0] \neq \emptyset) \leq p_0.
\]
\label{lemma-HUVV}
\end{lemma}

\reversemarginpar

The next lemma gives us control over the number of renewal marks that appear on the timeline of a vertex during a finite period of time. For any finite set $A$, we denote its cardinality by $|A|$.\\

\begin{lemma}
 Assume that $\mu\{0\} <1$ and let $\mathcal{R}$ be a renewal process with interarrival distribution $\mu$ started from some $\mathfrak{T} \leq 0$. Then, for each $p_0 > 0$ and $s > 0$ there exists $K_0 = K_0(p_0,s)$ such that uniformly on $\mathfrak{T}$ we have:
\[
\sup_{t \geq 0} P( \ | \ \mathcal{R} \cap [t,t+s] \ | > K_0) \leq p_0.
\]
\label{lemma2-bounded}
\end{lemma} 

Lemmas \ref{lemma-HUVV} and \ref{lemma2-bounded} together allow us to state a lemma that gives control over how close the cure marks of two adjacent vertices inside a finite time interval are from each other.

\begin{lemma}
\label{lemma3-goodcures}
Let $\mu$ be a continuous distribution on $[0,\infty)$ and let $\mathcal{R}$, $\tilde{\mathcal{R}}$ be two independent renewal processes with interarrival distribution $\mu$ started from some $\mathfrak{T} \leq 0$. Then given $p_1 > 0$ and $s>0$ there exists $w_1 = w_1(p_1,s) > 0$ such that uniformly on $\mathfrak{T}$ we have:
\[
\sup_{t \geq 0} P\Big( {\rm{d}}\left(\mathcal{R}\cap[t,t+s],\tilde{\mathcal{R}}\cap[t,t+s] \right) \leq w_1 \Big) \leq p_1,
\]
where ${\rm{d}}(A,B)=\inf\{|x-y|\colon x \in A, y \in B\}$ stands for the usual minimal distance between two subsets $A, B$ of the real line, understood as infinity if one of them is empty. 
\end{lemma}

Now the next two lemmas state that if we choose $\lambda$ large enough, the infection can survive inside specific space-time boxes with high probability. For both lemmas, consider the RCP($\mu$) where $\mu$ is a continuous distribution such that $\mu[0, b] = 1$. %\marginnote{\red \tiny vale a pena enunciar os lemas sem limitar o suporte da $\mu$?}

\vspace{0.2cm}

For convenience we set the following definition. 

\begin{definition}\label{crossing}
\noindent (a) Given bounded subsets of $\Z \times \mathbb R$,  $C$, $D$ and $H$, we say there is a crossing from $C$ to $D$ in $H$ if there exists a path $\gamma: [s,t] \ \rightarrow  H$  such that 
$(\gamma (s),s )  \in C$ and $(\gamma (t),t ) \in  D$.

\noindent (b) Given $b>0, t\geq 0$ and $x \in \Z$, consider the space-time boxes
$$
B_h(x,t)=[x,x+3] \times [t,t+b], \quad B_v(x,t)=[x,x+1] \times [t,t+3b], 
$$
and set
\begin{eqnarray*}
C_h(x,t)&=&\{ \text{there is a crossing from } (x,t) \text { to }\{x+3\} \times [t, t+b] \text { in } B_h(x,t)\},\\
C_v(x,t)&=&\{ \text{there is a crossing from } (y,t) \text { to }\{x,x+1\} \times \{t+3b\}  \text { in } B_v(x,t), y=x,x+1 \}.\\
%C_v(x,t)&=&\{ \text{there is a crossing from } \{(x,t),(x+1,t)\} \text { to }\{x,x+1\} \times \{t+3b\}  \text { in } B_v(x,t)\}.
\end{eqnarray*}
\end{definition}

\begin{lemma}
Fix $b > 0$ and $\epsilon \in (0,1)$. There exists $\lambda_h = \lambda_h(\epsilon) < \infty$ such that:
\[
\inf_{t \geq 0, x \in \Z} P(C_h(x,t)) \geq 1-\epsilon, \ \mbox{if} \ \lambda \geq \lambda_h.
\]
\label{lemmahb-bounded}
\end{lemma}

\begin{lemma}
Fix $b > 0$ and $\epsilon \in (0,1)$. There exists $\lambda_v = \lambda_v(\epsilon) < \infty$ such that:
\[
\inf_{t \geq 0, x \in \Z} P(C_v(x,t)) \geq 1-\epsilon, \ \mbox{if} \ \lambda \geq \lambda_v.
\]
\label{lemmavb-bounded}
\end{lemma}

We start by proving  the above lemmas.

\begin{proof}
\textit{(Lemma \ref{lemma2-bounded})} Let $(X_j)_{j \geq 1}$ be a sequence of i.i.d random variables with distribution $\mu$. Then, we have

\vspace{-0.25cm}
\begin{equation}
\sup_{t \geq 0} P(\ | \ \mathcal{R} \cap [t,t+s] \ | > K_0) \leq P\Big(\sum_{j=1}^{K_0} X_j < s\Big) = P\Big(\frac{1}{K_0}\sum_{j=1}^{K_0} X_j < \frac{s}{K_0}\Big).
\label{eqlgn-bounded2}
\end{equation}
Since the random variables $(X_j)_{j \geq 1}$ are non-negative and not identically null, the strong law of large numbers tells us that $\frac{1}{K}\sum_{j=1}^{K} X_j$ converges almost surely to $E[X_1] \in (0, \infty]$, as $K \to \infty$. It then follows at once from \eqref{eqlgn-bounded2} that for each  $p_0 > 0$ and $s > 0$, we can choose $K_0$ as in the statement.

\end{proof}

\begin{proof}
\textit{(Lemma \ref{lemma3-goodcures})} The uniformity on $t\geq 0$ (and consequently also for $\mathfrak{T}\leq 0$) follows at once from its validity in Lemmas \ref{lemma-HUVV} and \ref{lemma2-bounded}. 
Given $p_1 > 0$ we use Lemma \ref{lemma2-bounded} for $\mathcal{\tilde R}$ with $p_0=p_1/2$ and let $\tilde {K}_0=K_0(p_1/2,s)$. We then use Lemma  \ref{lemma-HUVV} with $p_0=p_1/{2\tilde K_0}$, letting $w_1= \frac{1}{2}w_0(\frac{p_1}{2\tilde{K}_0})$. It then follows that 
$$P\Big( {\rm{d}}\left(\mathcal{R}\cap[t,t+s],\tilde{\mathcal{R}}\cap[t,t+s] \right) \leq w_1 \Big) \leq \frac{p_1}{2\tilde{K}_0}\tilde K_0 + \frac{p_1}{2}=p_1.$$
\end{proof}

\begin{proof}
\textit{(Lemma \ref{lemmahb-bounded})} Fix $w_0 \in (0,b)$ satisfying Lemma \ref{lemma-HUVV} for $p_0 = 1- (1-\epsilon)^{\frac{1}{6}}$ and define the events $F_i = F_i(x,t)$ and $G_i = G_i(x,t)$ for $i=0,1,2$ as the following:
\begin{align*}
F_i &= \Big\{\mathcal{R}_{x+i} \cap \mbox{$[t,t+w_0]$} = \emptyset\Big\},\\
G_i &= \Big\{\mathcal{N}_{x+i,x+i+1} \cap \mbox{$[t+\frac{i w_0}{3},t+\frac{(i+1)w_0}{3}]$} \neq \emptyset\Big\}.
\end{align*}

Note that the occurrence of $\bigcap_{i=0}^2 (F_i \cap G_i)$ implies the occurrence of $C_h(x,t)$, as illustrated in Figure \ref{fig1-bounded-revised}. So, our next step will focus on finding some useful lower bounds on the probabilities of the events $F_i$ and $G_i$ (the latter one depending on $\lambda$). Applying Lemma \ref{lemma-HUVV} we have that, for $i=0,1,2$,
\begin{equation}
\label{F-bounded}
\inf_{t \geq 0, x \in \mathbb{Z}} P(F_i) = \inf_{t \geq 0} P(\mathcal{R} \cap [t,t+w_0] = \emptyset) \geq (1-\epsilon)^{\frac{1}{6}}.
\end{equation}

\begin{figure}[ht!]
\centering
\includegraphics[scale=0.4]{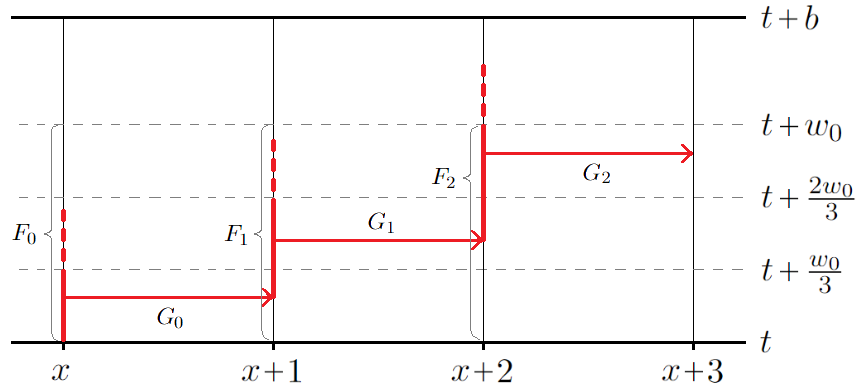}
\caption{Illustration of the events $F_i$ and $G_i$ occurring together, for $i=0,1,2$, implying the occurrence of $C_h(x,t)$. The red arrows indicate transmission marks (described by the events $G_0$, $G_1$ and $G_2$) and the thick red lines indicate a path without cure marks (consequence of the events $F_0$, $F_1$ and $F_2$).}
\label{fig1-bounded-revised}
\end{figure}

Now let $Y_{\lambda}$ denote a random variable with exponential distribution with rate $\lambda$. Since $w_0 > 0$, we can fix $\lambda_h$ such that $P(Y_{\lambda_h} \leq \frac{w_0}{3}) = (1-\epsilon)^{\frac{1}{6}}$, i.e., $\lambda_h = \frac{-3}{w_0}\log(1-(1-\epsilon)^{\frac{1}{6}})$. If $\lambda \geq \lambda_h$, by the Markov property of the  Poisson process, we have that for $i=0,1,2$,
\begin{equation}
\label{G-bounded}
\inf_{t \geq 0, x \in \mathbb{Z}} P(G_i) =  P\Big(Y_{\lambda} \leq \frac{w_0}{3}\Big) \geq (1-\epsilon)^{\frac{1}{6}}.
\end{equation}
But the events $F_0,G_0,F_1,G_1,F_2,G_2$ are all independent, and therefore the conclusion follows from \eqref{F-bounded} and \eqref{G-bounded}.
\end{proof}
 %\marginnote{\tiny \blue mexi um pouco sem alterar conteúdo}

\begin{proof}
\textit{(Lemma \ref{lemmavb-bounded})} Since $B_v(x,t)$ contains only two sites, the path has to keep jumping between these two sites until time $t+3b$, avoiding cure marks, as illustrated in Figure \ref{fig2-bounded}.

\begin{figure}[ht!]
\centering
\includegraphics[scale=0.3]{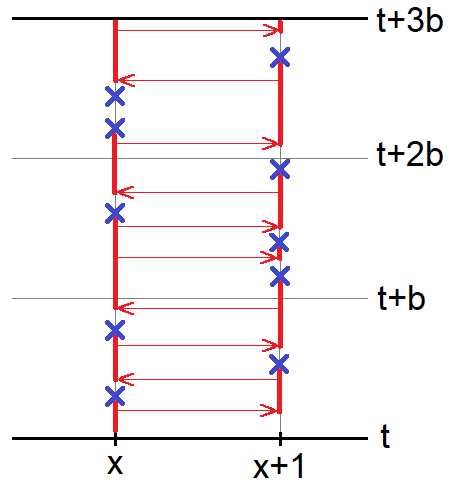}
\caption{Illustration of a vertical crossing inside the box $B_v(x,t)$. The red arrows indicate transmission marks, the blue X's indicate cure marks and the thick red lines indicate a path without cure marks.}
\label{fig2-bounded}
\end{figure}

Fix $p_0 = p_1 = \frac{1-\sqrt{1-\epsilon}}{5}$ and $s = 3b$. Apply Lemmas \ref{lemma-HUVV}, \ref{lemma2-bounded} and \ref{lemma3-goodcures} for $\mathcal{R}_x$ and $\mathcal{R}_{x+1}$, letting $\tilde{w}_0 = w_0(p_0)$, $\tilde{K}_0 = K_0(p_0,s)$ and $\tilde{w}_1 = w_1(p_1,s)$. Now defining by $H_t$ the intersection of the events

\[
\Big\{ \mathcal{R}_x \cap [t,t+\tilde{w}_0] = \emptyset,  \ \mathcal{R}_{x+1} \cap [t,t+\tilde{w}_0] = \emptyset \Big\},
\]

\[
\Big\{ | \ \mathcal{R}_x \cap [t,t+3b] \ | \leq \tilde{K}_0, \ | \ \mathcal{R}_{x+1} \cap [t,t+3b] \ | \leq \tilde{K}_0 \Big\},
\]

\[
\Big\{{\rm{d}}\Big(\mathcal{R}_x\cap[t,t+3b],\mathcal{R}_{x+1}\cap[t,t+3b] \Big) > \tilde{w}_1  \Big\},
\]

\vspace{0.2cm}

\noindent it follows from these three lemmas that $\inf_{t \geq 0} P(H_t) \geq \sqrt{1-\epsilon}$. The occurrence of $H_t$ guarantees that  a path with the desired crossing property can be constructed by suitably hopping between $x$ and $x+1$ at most $2 \tilde{K_0}$ times to reach time $t+3b$ and that it will always encounter an interval of length at least $\tilde{w} = \tilde{w}_0 \wedge \tilde{w}_1$ without cure marks to execute each hop. Fix $\lambda_v$ such that $P(Y_{\lambda_v} \leq \tilde{w}) = (1 - \epsilon)^{\frac{1}{4\tilde{K}_0}}$, i.e., $\lambda_v = \frac{-1}{\tilde{w}}\log(1-(1-\epsilon)^{\frac{1}{4\tilde{K}_0}})$, and hence it follows by the Markov property of the Poisson process that for any $\lambda \geq \lambda_v$,
\[
\inf_{t \geq 0,x \in \mathbb{Z}} P(C_v(x,t)) \geq \inf_{t \geq 0} P(H_t) \left[P(Y_{\lambda_v} \leq \tilde{w})\right]^{2\tilde{K}_0} \geq 1 - \epsilon.
\]

\end{proof}

Now we are finally ready to prove Theorem \ref{theo-bounded}.

\begin{proof}

\textit{(Theorem \ref{theo-bounded})} Given $x \in \mathbb{Z}$, $t > 0$ and $b > 0$ such that $\mu[0,b] = 1$, consider the space-time boxes $K(x,t) = [x,x+3] \times [t,t+3b]$, which will be called blocks in our proof. The block $K(x,t)$ is said to be \textit{good} if $C_h(x,t) \cap C_h(x,t+2b) \cap C_v(x,t) \cap C_v(x+2,t)$ occurs, and \textit{bad}  otherwise. That is, a block will be called good when four specific smaller boxes inside it are crossed in the sense of the corresponding events $C_h$ and $C_v$. This is illustrated on the left side image of Figure \ref{fig3}, in which $K(0,0)$, $K(2,0)$ and $K(0,3b)$ are good blocks and $K(2,3b)$ is a bad block.

Let's consider a graph $G$ in $\mathbb{Z}^2$ with oriented edges $ \langle z,z + (1,0) \rangle$ and $\langle z,z + (0,1) \rangle$. Define a percolation system on $G$ related to the RCP($\mu$) by stating that the vertex $(x,y)$ is open if $K(2x,3yb)$ is a good block in the RCP($\mu$), and closed otherwise. Moreover, each edge is open if both endvertices are open. Note that in this construction, the blocks are fitted together (and overlap in the horizontal direction) allowing us to concatenate paths of adjacent good blocks, and implying that when the  oriented edge model percolates, the associated RCP($\mu$) will survive with positive probability. See Figure \ref{fig3} for an example of this  percolation model being constructed. Given $\epsilon > 0$, it follows from Lemmas \ref{lemmahb-bounded}, \ref{lemmavb-bounded} and union bound that the probability of a site of $G$ being open is at least $1 - \epsilon$ if we take $\lambda \geq \lambda_h(\frac{\epsilon}{4}) \vee \lambda_v(\frac{\epsilon}{4})$.

\vspace{-0.1cm}

\begin{figure}[h]
    \centering
    \includegraphics[scale=0.3]{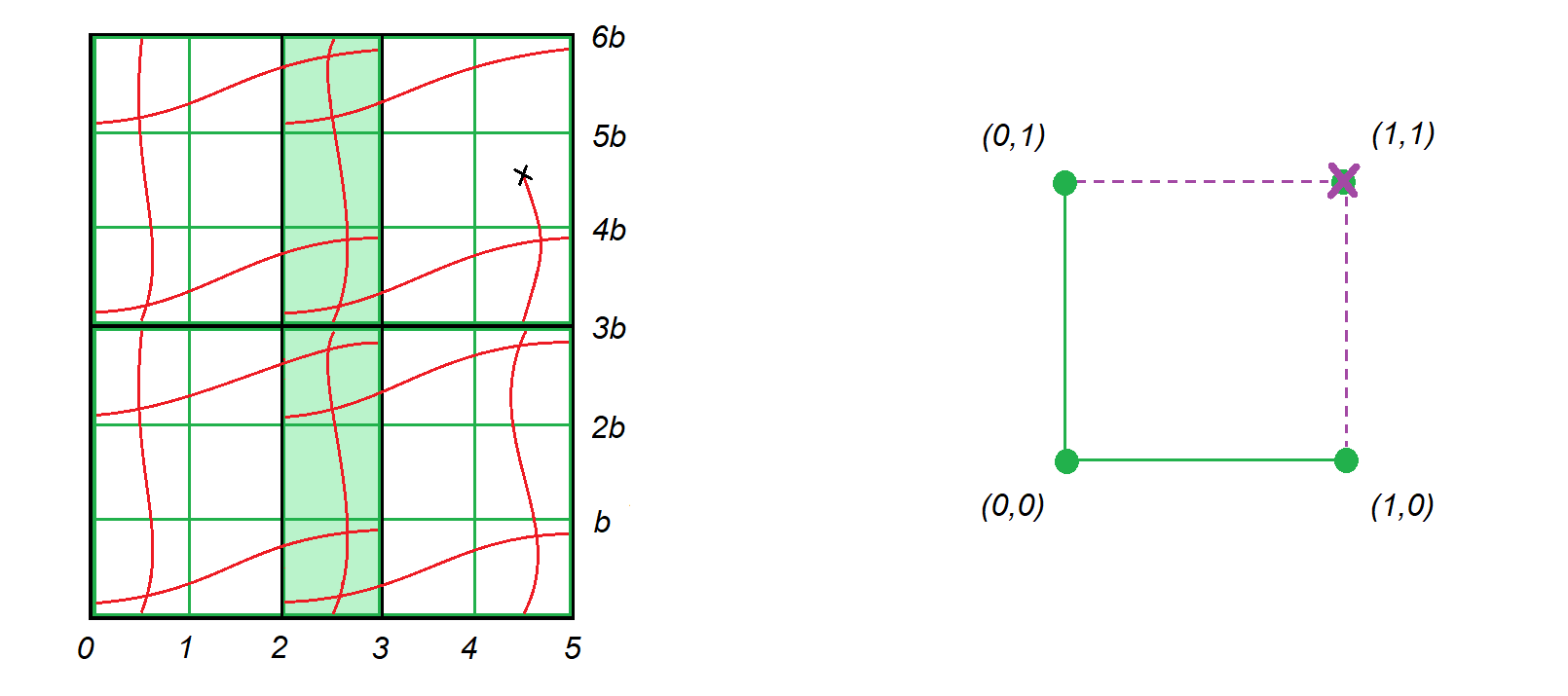}
    \caption{Example of the relation between the RCP($\mu$) and the  percolation model on $G$, mapping four blocks into four vertices. On the left side image, each curved red line does not represent a path itself, but symbolizes that the associated horizontal or vertical box are being crossed, with exception of the one that reach an $\times$ symbol, and the painted green boxes in the middle are indicating the regions where two blocks overlaps each other.  On the right side image we have the corresponding sites and edges of $G$ where the green dots (lines) indicate the open sites (edges) and the purple $\times$ symbol (dashed line) indicates the closed sites (edges).}
    \label{fig3}
\end{figure}

\vspace{0.1cm}

Now we will argue that in this  oriented percolation model, the edges that do not share endvertices are independent. It is a consequence of three facts:

\begin{itemize}
\item The independence properties of cure (transmission) marks on (from) different vertices.

\vspace{0.1cm}

    \item  The transmission marks are determined by Poisson processes.

\vspace{0.1cm}

    \item The time between two consecutive cure marks in a vertex is bounded by $b$, which means that any time interval of length $b$ will have at least one cure mark in every vertex. It implies that the cure marks after time $t+b$ are independent of the cure marks before time $t$, for any $t \geq 0$. 
\end{itemize}

\vspace{0.1cm}

It follows that we have a finite range dependent bond model. Applying ~\cite[Theorem 0.0 (ii)]{LSS} and choosing  $\epsilon > 0$ small enough in the beginning, it follows that the RCP survives with positive probability if 
$\lambda \geq \lambda_h(\frac{\epsilon}{4}) \vee \lambda_v(\frac{\epsilon}{4})$.
\end{proof}
%$\lambda \geq \lambda_h \vee \lambda_v$.

\section{$\mu$ with decreasing hazard rate}
\label{dfr}

This section is devoted to proving Theorem \ref{theo-dfr}. Again, we start by stating and proving some useful lemmas.  Analogously to those in Section \ref{bounded}, they provide suitable lower bounds for the probability of crossing events. The difference is that taking advantage of the assumption on the hazard rate, this can be done uniformly on the history of the renewal processes.

The first two lemmas guarantee that for any fixed $p>0$, we can find a fixed non-trivial interval so that, uniformly over the history of $\mathcal{R}_x$ up to time $t$, the conditional probability of no mark of the renewal process in this interval will be at least $1-p$.

\begin{lemma}
Consider $\mu$ satisfying the assumptions of Theorem \ref{theo-dfr} and let $\mathcal{R}$ be a renewal process with interarrival distribution $\mu$ started from some $\mathfrak{T} \leq 0$. Then, given $p_0 > 0$ there exists $w_0 = w_0(p_0) > 0$ such that uniformly on $\mathfrak{T}$ we have: 
\[
\sup_{t \geq 0} \sup_{k \geq 0} P \left( \left. \mathcal{R} \cap [t+k,t+k+w_0] \neq \emptyset \ \right| \ \mathcal{R} \cap [t,t+k] = \emptyset \right) \leq p_0.
\]
\label{lemma1-dfr}
\end{lemma}

\begin{lemma}
Consider $\mu$ satisfying the assumptions of Theorem \ref{theo-dfr} and let $\mathcal{R}$ be a renewal process with interarrival distribution $\mu$ started from some $\mathfrak{T} \leq 0$. 
Given $t \geq 0$, let $A_t$ denote the time elapsed since the last renewal of $\mathcal{R}$ previous to time $t$ (letting $A_t = t - \mathfrak{T}$ if $\mathcal{R} \cap [\mathfrak{T},t)$ is empty). Then, for any $p_0 > 0$, there exists $w_0 = w_0(p_0) > 0$ such that uniformly on $\mathfrak{T}$ we have
\[
\sup_{t \geq 0} \sup_{0 < a \leq t-\mathfrak{T}} P \left( \left. \mathcal{R} \cap [t,t+w_0] \neq \emptyset \ \right| \ A_t = a \right) \leq p_0.
\]
\label{lemma2-dfr}
\end{lemma}

Lemma \ref{lemma2-dfr} allows us to have some control on how close the cure marks of two adjacent vertices inside a finite time interval  can be from each other, uniformly on the history of the cure marks in these vertices. From now on, let $A_{x,t}$ be the time elapsed since the last renewal of $\mathcal{R}_x$ previous to time $t$ (letting $A_{x,t} = t - \mathfrak{T}$ if $\mathcal{R}_x \cap [\mathfrak{T},t)$ is empty).

\begin{lemma}
\label{goodcures-dfr}
Let $\mu$ satisfy the assumptions of Theorem \ref{theo-dfr} and let $\mathcal{R}_0$, $\mathcal{R}_1$ be two independent renewal processes with interarrival distribution $\mu$ started from some $\mathfrak{T} \leq 0$. Then, given $p_1 > 0$ and $s>0$ there exists $w_1 = w_1(p_1,s) > 0$ such that uniformly on the triplet $\{t,A_{0,t},A_{1,t}\}$ for $t \geq 0$, we have
\[
P\Big( {\rm{d}}\left(\mathcal{R}_0\cap[t,t+s],\mathcal{R}_1\cap[t,t+s] \right) \leq w_1 \Big| A_{0,t},A_{1,t}\Big) \leq p_1, \ \mbox{almost surely.}
\]
\end{lemma}

The next two lemmas state that if we choose $\lambda$ large enough, the infection can survive inside specific space-time boxes with high probability, uniformly on the history of the renewal processes. For both lemmas, consider the RCP($\mu$) with $\mu$ being defined as in Theorem \ref{theo-dfr} and recall Definition \ref{crossing}.

\begin{lemma}
Let $C_h(x,t)$ be as in Definition \ref{crossing} with $b=1$ and let $\epsilon \in (0,1)$.
%Fix $b=1$ and $\epsilon \in (0,1)$. 
There exists $\lambda_h = \lambda_h(\epsilon) < \infty$ such that, uniformly on the $5-$tuple $\{t, A_{x,t}, A_{x+1,t}, A_{x+2,t}, A_{x+3,t}\}$ for $t \geq 0$, we have 
\[
P\Big(C_h(x,t)\Big|(A_{x+i,t})_{i=0}^{3}\Big) \geq 1-\epsilon, \ \mbox{almost surely, if} \ \lambda > \lambda_h.
\]
\label{lemmahb-dfr}
\end{lemma}

\begin{lemma}
 Let $C_v(x,t)$ be as in Definition \ref{crossing} with $b=1$ and let $\epsilon \in (0,1)$.
%Fix $b=1$ and $\epsilon \in (0,1)$. 
There exists $\lambda_v = \lambda_v(\epsilon) < \infty$ such that, uniformly on the triplet $\{t, A_{x,t}, A_{x+1,t}\}$ for $t \geq 0$, we have 
\[
P\Big(C_v(x,t)\Big|A_{x,t},A_{x+1,t}\Big) \geq 1-\epsilon, \ \mbox{almost surely, if} \ \lambda > \lambda_v.
\]
\label{lemmavb-dfr}
\end{lemma}

We now prove the lemmas.

\begin{proof}
\textit{(Lemma \ref{lemma1-dfr})} 
Let $(X_j)_{j \geq 0}$ be a sequence of i.i.d. random variables with distribution $\mu$, common density $f$ and common distribution function $F$, such that $X_0$ describes the time until the first renewal mark of $\mathcal{R}$ and $X_j$, for $j \geq 1$, describe the time between the $j$-th and $(j+1)$-th renewal marks of $\mathcal{R}$. In addition, denote by $(F^{(j)})_{j \geq 0}$ the distribution function of $X_0 + X_1 + \ldots + X_j$, by $B_t$ the residual life process of $\mathcal{R}$ at time $t$, i.e, the remaining time between $t$ and the next renewal mark of $\mathcal{R}$, and by $E_j$ the event $\{$there are exactly $j$ renewal marks of $\mathcal{R}$ up to time $t\}$. Then, for any $t \geq 0$, $k>0$ and $w>0$, we have that
\[
P \left( \left. \mathcal{R} \cap [t+k,t+k+w] \neq \emptyset \ \right| \ \mathcal{R} \cap [t,t+k] = \emptyset \right) = \sum_{j=0}^\infty P(B_t \leq k+w | B_t > k, E_j)P(E_j)
\]
\[
= P(E_0)P(X_0 \leq t+k+w | X_0 > t+k) + \sum_{j=1}^\infty P(E_j) \int_{0}^t P(B_{t-s} \leq k+w | B_{t-s} > k)dF^{(j-1)}(s) \ \
\]
\[
= P(E_0)\frac{P(X \in [t+k,t+k+w])}{P(X \geq t+k)} + \sum_{j=1}^\infty P(E_j) \int_{0}^t P(X_j \leq t+k+w-s | X_j > t+k-s)dF^{(j-1)}(s)
\]
\[
= P(E_0) \int_{t+k}^{t+k+w}\frac{f(u)}{1-F(t+k)} du + \sum_{j=1}^\infty P(E_j) \int_{0}^t \left[ \int_{t+k-s}^{t+k+w-s} \frac{f(u)}{1-F(t+k-s)} du \right] dF^{(j-1)}(s)  \ \ \
\]

\begin{equation}
\leq P(E_0) \int_{t+k}^{t+k+w}\frac{f(u)}{1-F(u)} du + \sum_{j=1}^\infty P(E_j) \int_{0}^t \left[ \int_{t+k-s}^{t+k+w-s} \frac{f(u)}{1-F(u)} du \right] dF^{(j-1)}(s). \ \ \ 
\label{eqn1}
\end{equation}

Since the hazard rate of $\mu$ is decreasing in $t$, we have that both integrals of $\frac{f(u)}{1-F(u)}$ in \eqref{eqn1} are smaller than or equal to 
\[
\int_{0}^{w} \frac{f(u)}{1-F(u)} du,
\]

\vspace{0.2cm}
\noindent and then we easily see that \eqref{eqn1} is bounded above by
\[
\frac{1}{1-F(w)} \int_{0}^{w}f(u)du \left[ P(E_0) + \sum_{j=1}^\infty P(E_j) \int_{0}^\infty dF^{(j-1)}(s) \right] 
\]
\[
= \frac{F(w)}{1-F(w)} \sum_{j=0}^\infty P(E_j) = \frac{F(w)}{1-F(w)}. \qquad \qquad \qquad \qquad \qquad \qquad
\] 

\vspace{0.1cm}

Since $F$ is a continuous distribution function on $[0,\infty)$, 
we can take $w_0 = w_0(p_0) > 0$ such that $F(w_0) \leq \frac{p_0}{2}$ and then:
\[
\sup_{t \geq 0}  \sup_{k \geq 0} P \left( \left. \mathcal{R} \cap [t+k,t+k+w_0] \neq \emptyset \ \right| \ \mathcal{R} \cap [t,t+k] = \emptyset \right) \leq \frac{F(w_0)}{1-F(w_0)} \leq \frac{p_0}{2 - p_0} \leq p_0. \vspace{-0.4cm}
\]
\end{proof}

\vspace{0.1cm}

\begin{proof}
\textit{(Lemma \ref{lemma2-dfr})} The statement is equivalent to the existence of $w_0 = w_0(p_0)$ such that, uniformly on $\mathfrak{T}$, we have
\[
\sup_{a \geq 0} \frac{\mu([a,a+w_0])}{\mu([a,\infty))} \leq p_0.
\]

\vspace{0.1cm}

But for any fixed $w>0$ and $a>0$,
\[
\frac{\mu([a,a+w])}{\mu([a,\infty))} = \int_{a}^{a+w} \frac{f(u)}{1-F(a)}du \leq \int_{a}^{a+w} \frac{f(u)}{1-F(u)}du.
\]

\vspace{0.1cm}

Since the hazard rate of $\mu$ is decreasing, the last integral above is bounded by
\[
\int_0^w \frac{f(u)}{1-F(u)}du \leq \frac{1}{1-F(w)}\int_0^w f(u)du = \frac{F(w)}{1-F(w)},
\]

\vspace{0.1cm}

\noindent and the conclusion follows as in Lemma \ref{lemma1-dfr}.
\end{proof}

\begin{remark}
The bound on the number of renewal marks of $\mathcal{R}_x$ in a finite time interval, that was obtained in Lemma \ref{lemma2-bounded}, also holds almost surely when we condition to $A_{x,t}$, uniformly on $\{t,A_{x,t}\}$ (enlarging $K_0$ by one if necessary). It occurs because after the first renewal mark of $\mathcal{R}_x$ inside $[t,t+s]$ appears, we are able to forget all information of $\mathcal{R}_x$ previous to time $t$. 
\end{remark}

\begin{proof}
\textit{(Lemma \ref{goodcures-dfr})} 
Given $p_1 > 0$ we use Lemma \ref{lemma2-bounded} for $\mathcal{R}_1$ with $p_0=p_1/2$ and let $\tilde {K}_0=K_0(p_1/2,s)$. We then use Lemma  \ref{lemma2-dfr} with $p_0=p_1/{2\tilde K_0}$, letting $w_1= \frac{1}{2}w_0(\frac{p_1}{2\tilde{K}_0})$. It then follows that 
$$P\Big( {\rm{d}}\left(\mathcal{R}_0\cap[t,t+s],\mathcal{R}_1\cap[t,t+s] \right) \leq w_1 \Big| A_{0,t},A_{1,t}\Big) \leq \frac{p_1}{2\tilde{K}_0}\tilde K_0 + \frac{p_1}{2}=p_1, \ \mbox{almost surely.}$$

The uniformity on $t\geq 0$ follows at once from its validity in Lemmas \ref{lemma2-bounded} and \ref{lemma2-dfr}; and the uniformity on $A_{0,t}, A_{1,t}$ follows from Lemma \ref{lemma2-dfr}. 
\end{proof} 

\begin{proof}
\textit{(Lemma \ref{lemmahb-dfr})} Fix $w_0 \in (0,1)$ satisfying Lemma \ref{lemma2-dfr} for $p_0 = 1- (1-\epsilon)^{\frac{1}{6}}$ and define the events $F_i = F_i(x,t)$ and $G_i = G_i(x,t)$ for $i=0,1,2$ as the following:
\begin{align*}
F_i &= \Big\{\mathcal{R}_{x+i} \cap \mbox{$[t,t+w_0]$} = \emptyset\Big\},\\
G_i &= \Big\{\mathcal{N}_{x+i,x+i+1} \cap \mbox{$[t+\frac{i w_0}{3},t+\frac{(i+1)w_0}{3}]$} \neq \emptyset\Big\}.
\end{align*}

Note that the occurrence of $\bigcap_{i=0}^2 (F_i \cap G_i)$ implies the occurrence of $C_h(x,t)$, as in the proof of Lemma \ref{lemmahb-bounded} (see Figure \ref{fig1-bounded-revised}). Applying Lemma \ref{lemma2-dfr} we have that uniformly on $\{t,A_{x+i,t}\}$, for $i=0,1,2$ and $t \geq 0$, 
\begin{equation}
\label{F}
P(F_i|A_{x+i,t}) = P(\mathcal{R}_x \cap [t,t+w_0] = \emptyset | A_{x,t}) \geq (1-\epsilon)^{\frac{1}{6}}, \ \mbox{almost surely}.
\end{equation}

Now let $Y_{\lambda}$ denote a random variable with exponential distribution with rate $\lambda$. Since $w_0 > 0$, we can fix $\lambda_h$ such that $P(Y_{\lambda_h} \leq \frac{w_0}{3}) = (1-\epsilon)^{\frac{1}{6}}$. If $\lambda \geq \lambda_h$, by the Markov property of the  Poisson process, we have that for $i=0,1,2$ and any $t \geq 0$,
\begin{equation}
\label{G}
P(G_i) =  P\Big(Y_{\lambda} \leq \frac{w_0}{3}\Big) \geq (1-\epsilon)^{\frac{1}{6}}.
\end{equation}
But the events $F_0,G_0,F_1,G_1,F_2,G_2$ are all independent, and therefore the conclusion follows from \eqref{F} and \eqref{G}.
\end{proof}

\begin{proof}
\textit{(Lemma \ref{lemmavb-dfr})}  Since $B_v(x,t)$ contains only two sites, the path might have to jump between them until time $t+3$, avoiding cure marks, as in the proof of Lemma \ref{lemmavb-bounded} (see Figure \ref{fig2-bounded}).

Fix $p_0 = p_1 = \frac{1-\sqrt{1-\epsilon}}{5}$ and $s = 3$. Apply Lemmas \ref{lemma2-bounded}, \ref{lemma2-dfr} and \ref{goodcures-dfr} for $\mathcal{R}_x$ and $\mathcal{R}_{x+1}$, letting $\tilde{w}_0 = w_0(p_0)$, $\tilde{K}_0 = K_0(p_0,s)$ and $\tilde{w}_1 = w_1(p_1,s)$. Defining $H_t$ as the intersection of the events

\[
\Big\{ \mathcal{R}_x \cap [t,t+\tilde{w}_0] = \emptyset,  \ \mathcal{R}_{x+1} \cap [t,t+\tilde{w}_0] = \emptyset \Big\},
\]

\[
\Big\{ | \ \mathcal{R}_x \cap [t,t+3] \ | \leq \tilde{K}_0, \ | \ \mathcal{R}_{x+1} \cap [t,t+3] \ | \leq \tilde{K}_0 \Big\},
\]

\[
\Big\{{\rm{d}}\Big(\mathcal{R}_x\cap[t,t+3],\mathcal{R}_{x+1}\cap[t,t+3] \Big) > \tilde{w}_1  \Big\},
\]

\vspace{0.2cm}

\noindent it follows from these three lemmas that uniformly on the triplet $\{t,A_{x,t},A_{x+1,t}\}$, for $t \geq 0$, we have $P(H_t|A_{x,t},A_{x+1,t}) \geq \sqrt{1-\epsilon}$ almost surely. The occurrence of $H_t$ guarantees that a path with the desired crossing property can be constructed by suitably hopping between $x$ and $x+1$ at most $2 \tilde{K_0}$ times to reach time $t+3$ and that it will always encounter an interval of length at least $\tilde{w} = \tilde{w}_0 \wedge \tilde{w}_1$ without cure marks to execute each hop. Fix $\lambda_v$ such that $P(Y_{\lambda_v} \leq \tilde{w}) = (1 - \epsilon)^{\frac{1}{4\tilde{K}_0}}$, and hence it follows by the Markov property of the Poisson process that for any $\lambda \geq \lambda_v$, uniformly on the triplet $\{t,A_{x,t},A_{x+1,t}\}$, for $t \geq 0$, we have
\[
P(C_v(x,t)|A_{x,t},A_{x+1,t}) \geq P(H_t|A_{x,t},A_{x+1,t}) \left[P(Y_{\lambda_v} \leq \tilde{w})\right]^{2\tilde{K}_0} \geq 1 - \epsilon, \ \mbox{almost surely.}
\]

\end{proof}

\begin{corollary}
\label{coroboxes-dfr}

Let $\mathcal{F}_t$ be the filtration generated by the processes $(\mathcal{R}_x)_{x \in \mathbb{Z}}$ and $(\mathcal{N}_{x,y})_{(x,y) \in \mathcal{E}}$. The uniform bounds obtained in Lemmas \ref{lemmahb-dfr} and \ref{lemmavb-dfr} also hold almost surely if we condition on $\mathcal{F}_t$.

\end{corollary}

\begin{proof}
\textit{(Corollary \ref{coroboxes-dfr})} The space-time boxes $B_h(x,t)$ and $B_v(x,t)$, related to the events $C_h(x,t)$ and $C_v(x,t)$, are entirely above the timeline $t$. So, these events are independent of the transmission marks up to time $t$ and depend on the cure marks up to time $t$ only through the position of the last cure mark previous to time $t$ in some vertices.
\end{proof}

Now we are finally ready to present the proof of Theorem \ref{theo-dfr}.

\begin{proof}

\textit{(Theorem \ref{theo-dfr})} Here we will construct a percolation system related to the RCP($\mu$) in the same way that we did in the proof of Theorem \ref{theo-bounded}, taking $b=1$. Given $x \in \mathbb{Z}$ and $t > 0$, consider the space-time boxes hereby called blocks, $K(x,t) = [x,x+3] \times [t,t+3]$. Each block $K(x,t)$ is said to be \textit{good} if $C_h(x,t) \cap C_h(x,t+2) \cap C_v(x,t) \cap C_v(x+2,t)$ occurs and \textit{bad} otherwise. See  for instance the left side image of Figure \ref{fig3}.   

We consider a graph $G$ in $\mathbb{Z}^2$ with oriented edges $\langle z,z + (1,0) \rangle$ and $ \langle z,z + (0,1) \rangle$, and define a percolation system on $G$ related to the RCP($\mu$) by stating that the vertex $(x,y)$ is open if $K(2x,3y)$ is a good block in the RCP($\mu$) and closed otherwise. Then, each edge is open if both endvertices are open. Note that this construction allows us to concatenate paths of adjacent good blocks, implying that when the oriented edge model percolates, the associated RCP($\mu$) will survive with positive probability. See Figure \ref{fig3}.

This percolation model exhibits long-range dependencies along the vertical direction (since the positions of the cure marks in the associated RCP are non-Markovian) and finite range dependency in the first coordinate (since horizontally adjacent blocks overlap each other). But, given $\epsilon > 0$, Lemmas \ref{lemmahb-dfr}, \ref{lemmavb-dfr} (and Corollary \ref{coroboxes-dfr}) and union bounds imply that if we take $\lambda \geq \lambda_h(\frac{\epsilon}{4}) \vee \lambda_v(\frac{\epsilon}{4})$, then the probability that a site of $G$ is open, given any possible configurations of the other sites with lower vertical coordinate, is bounded from below by $1- \epsilon$. Applying ~\cite[Theorem 0.0 (i)]{LSS} we can conclude that this percolation model dominates a Bernoulli random field with density $\rho$ which can be arbitrarily close to $1$, provided we take $\epsilon$ small enough. As a consequence, we have that the RCP($\mu$) survives with positive probability for any $\lambda \geq \lambda_h(\frac{\epsilon}{4}) \vee \lambda_v(\frac{\epsilon}{4})$.\\ 
\end{proof}


\begin{thebibliography}{99}

\bibitem{Durrett}
R. Durrett: Ten Lectures on particle systems. (Ecole d'Et\'e de Probabilit\'es de Saint-Flour XXIII, 1993) {\it Lecture Notes in Math.}, {\bf 1608}, 97--201, Springer, Berlin (1995)

\bibitem{F}
L. R. Fontes: Contact Process under renewal cures. An overview of recent results. {\it Matemática Contemporânea}, {\bf 58}, 234--263  (2023).


\bibitem{FMMV1}
L. R. Fontes,  D.H.U. Marchetti, T.S. Mountford, M. E. Vares: Contact process under renewals I. {\it Stoch. Proc. Appl.}, {\bf 129}(8), 2903--2911 (2019).

\bibitem{FGS}
 L. R. Fontes, P. Gomes, R. Sanchis: Contact process under heavy-tailed renewals on finite graphs. {\it Bernoulli}, {\bf 27}(3), 1745--1763 (2020).

\bibitem{FMV2}
L. R. Fontes, T.S. Mountford, M. E. Vares: Contact process under renewals II. {\it Stoch. Proc. Appl.}, {\bf 130}(2), 1103--1118 (2020).

\bibitem{FMUV}
L. R. Fontes, T.S. Mountford, D. Ungaretti, M. E. Vares: Renewal Contact Processes: phase transition and survival. {\it Stoch. Proc. Appl.}, {\bf 161}, 102--136 (2023).

\bibitem{Harris}
T. E. Harris: Contact interactions on a lattice.  {\it The Annals of Probability}, {\bf 2}(6) 969--988 (1974)

\bibitem{Harris2}
T. E. Harris: Aditive Set-Valued Markov Process and Graphical Methods.  {\it The Annals of Probability}, {\bf 6}(3) 355--378 (1978)

\bibitem{HUVV}
M. Hilário, D. Ungaretti, D. Valesin, M. E. Vares: Results on the contact process with dynamic edges or under renewals. {\it Electron. J. Probab.} {\bf 27}, 1--31 (2022). 

\bibitem{Liggett}
T. M. Liggett: Interacting Particle Systems. {\it Grundlehren der Mathematischen Wissenschaften} {\bf 276}, New York: Springer (1985)

\bibitem{LSS}
T. Liggett, R. Schonmann, A. Stacey: Domination by product measures. {\it The Annals of Probability}, {\bf 25}(1), 71-95 (1997)

\bibitem{LR}
A. Linker, D. Remenik: The contact process with dynamic edges on $\mathbb{Z}$. {\it Electron. J. Probab.} {\bf 25}, 1--21 (2020)  

\end{thebibliography}
\end{document}